\documentclass[conference,letterpaper]{IEEEtran}

\usepackage[utf8]{inputenc} 
\usepackage[T1]{fontenc}
\usepackage{url}
\usepackage{ifthen}
\usepackage{cite}
\usepackage[cmex10]{amsmath} 

\usepackage{amssymb}
\usepackage{amsfonts, bm}
\usepackage{mathtools}
\usepackage{color}

\newtheorem{proposition}{Proposition}
\newtheorem{lemma}{Lemma}
\newtheorem{theorem}{Theorem}
\newtheorem{conjecture}{Conjecture}
\newenvironment{proof}{\par\noindent\textit{Proof: }}{\hfill$\blacksquare$\par}

\newcommand{\blue}[1]{\textcolor{black}{#1}}
\newcommand{\setR}{\mathbb{R}}

\DeclareMathOperator{\tr}{tr}
 
\newcommand{\tp}{\mathsf{T}}
\newcommand{\argmax}{\mathop{\rm arg~max}\limits}

\DeclareMathOperator{\erfcx}{erfcx}

\begin{document}
\title{Empirical Bayes Estimation for Lasso-Type Regularizers: Analysis of Automatic Relevance Determination} 

\author{%
 \IEEEauthorblockN{Tsukasa Yoshida and Kazuho Watanabe}
\IEEEauthorblockA{Department of Computer Science and Engineering\\
                   Toyohashi University of Technology\\
                   Japan\\
                   Email: yoshida.tsukasa.fp@tut.jp}
}

\maketitle


\begin{abstract}
This paper focuses on linear regression models with non-conjugate sparsity-inducing regularizers such as lasso and group lasso.
Although the empirical Bayes approach enables us to estimate the regularization parameter, little is known on the properties of the estimators.
In particular, many aspects regarding the specific conditions under which the mechanism of automatic relevance determination (ARD) occurs remain unexplained.
In this paper, we derive the empirical Bayes estimators for the group lasso regularized linear regression models with limited parameters.
It is shown that the estimators diverge under a specific condition, giving rise to the ARD mechanism.
We also prove that empirical Bayes methods can produce the ARD mechanism in general regularized linear regression models and clarify the conditions under which models such as ridge, lasso, and group lasso can do so.
\end{abstract}

The full version of this paper, including the Appendix, is accessible at \url{https://arxiv.org/abs/2501.11280}.

\section{Introduction}
In regression problems, regression functions overfit to the training dataset if the selected model is too complex for the data. 
One of the solutions to suppress overfitting is regularization.
Regularization methods impose a penalty on the complexity of the regression function.
In particular, $L_2$ and $L_1$ regularizers are typical, which regard the norm of model parameters as the penalty.
The linear regression models regularized by $L_2$ and $L_1$ norm are known as ridge and lasso, respectively~\cite{bishop2006pattern,tibshirani1996regression}. 
Since $L_1$ regularization induces sparsity in the parameters, it finds wide applications in sparse modeling. Additionally, many variants of lasso have been proposed, such as elastic net~\cite{zou2005regularization}, group lasso~\cite{yuan2006model}, fused lasso~\cite{tibshirani2005sparsity}, and total variation~\cite{rudin1992nonlinear}.

In these regularization methods, however, the learning result significantly changes depending on the value of the regularization parameter, which controls the strength of regularization in a loss function.
One procedure for estimating hyperparameters is empirical Bayes estimation, which maximizes the marginal likelihood of the hyperparameters~\cite{bernardo1994bayesian, gelman1995bayesian}. 
The ridge model generally does not sparsify the parameters, but does when combined with the empirical Bayes method for setting the regularization parameter.
This sparsification method is known as automatic relevance determination (ARD)~\cite{mackay1992bayesian,tipping2001sparse}.

However, the empirical Bayes estimators of most non-conjugate models, including lasso, group lasso, and most of the practical regularizations, have yet to be derived entirely because non-conjugacy between the likelihood and the prior distribution makes the integration of the marginal likelihood intractable. 
Hence, in practice, approximate empirical Bayes methods, such as variational approximations, are applied to its computation~\cite{girolami2001variational,seeger2008bayesian,yasuda2022empirical,omae2022approximate}.
Since approximate estimation methods perform iterative optimization, the behaviors of the exact and approximate empirical Bayes solutions have yet to be fully elucidated.
Additionally, even in the case of ridge with conjugacy, the empirical Bayes estimators have not been fully discussed, while the estimators have been explicitly obtained only for some specific models~\cite{nakajima2004Generalization}. This solution clearly explains the scheme of ARD. In this line of studies, \cite{nakajima2006generalization,nakajima2007variational} analyzed the asymptotic behavior of ridge regression and its extended models.

Cross-validation is a widely used framework for estimating hyperparameters. Theoretical analysis of cross-validation has been conducted, 
and its behavior has been studied by 
\cite{obuchi2016cross,obuchi2019cross}. 
The prediction error defines the objective function in cross-validation, whereas for empirical Bayes estimation, it encompasses the (negative log) marginal likelihood, which can be interpreted as the cumulative prediction error~\cite{levin1990statistical}. Since these criteria correspond to the Akaike information criterion (AIC) and the Bayesian information criterion (BIC), each is expected to exhibit distinct characteristics.

Recently, an exact empirical Bayes estimator for the lasso in a one-parameter model has been derived by \cite{yoshida2018empirical}. 
This estimator for lasso includes an ARD scheme similar to that observed in ridge.
What kind of models give rise to ARD?
Although ARD is well-known, the conditions for its occurrence remain largely unexplained.
To understand the general characteristics of empirical Bayes methods and ARD,
this study aims to derive empirical Bayes estimators of regularization parameters for regularized linear regression models and analyze them from the perspective of ARD. 
As a result, we elucidate some of the conditions under which ARD occurs.
The following are three main contributions:
\begin{enumerate}
    \item We rigorously derive the empirical Bayes estimator for the group lasso regularized linear regression model. The derivation is performed for models with a limited number of parameters.
    \item We derive the empirical Bayes estimator for the general regularized linear regression model, contingent upon specified conditions. This estimator encompasses those for ridge, lasso, and group lasso, for each of which exact solutions have been derived. Furthermore, this result clarifies the conditions under which the ARD mechanism arises.
    \item We numerically evaluate the marginal likelihood in examples of ridge, lasso, and group lasso to investigate its properties. 
    This analysis suggests that the conditions assumed for the empirical Bayes estimator introduced earlier are likely satisfied, implying that the empirical Bayes estimators for ridge, lasso, and group lasso can be unified.
\end{enumerate}


\section{Problem Statement}
This chapter describes the model constituting the core focus of this paper and the empirical Bayes estimation.

\subsection{Regularized Linear Regression Model}
Given the dataset $\{(\bm{x}_1, y_1), \ldots, (\bm{x}_n, y_n)\}\ ((\bm{x}_i, y_i)\in \setR^d\times \setR, i= 1,\ldots,n)$, 
we consider the following Gaussian model with the parameter $\bm{w} \in \setR^m$ and the prior distribution of $\bm{w}$ with a hyperparameter $\lambda \in \setR_{> 0}$
\footnote{
  This model turns out to be equivalent to the Gauss model $\mathcal{N}(\bm{w}^\tp\bm{\phi}(\bm{x}), \sigma^2)$ with the known variance $\sigma^2$ by standardizing variable transformation, $y':= y/\sigma$ and $\phi':= \phi/\sigma$.
}:
\begin{align}
    p(y|\bm{x}, \bm{w}) &:= \frac{1}{\sqrt{2\pi}} e^{-\frac{1}{2}(y - \bm{w}^\tp\bm{\phi}(\bm{x}))^2}, &
    p(\bm{w}|\lambda) &:= \frac{1}{C} e^{-\lambda h(\bm{w})}, 
    \label{eq: model}
\end{align}
where $\bm{\phi}(\bm{x}) = (\phi_1(\bm{x}), \ldots, \phi_m(\bm{x}))^\tp$ is the feature vector, $\phi_j : \setR^d \to \setR\ (j=1, \ldots, m)$ are features of $\bm{x}$, and $C$ is the normalizing constant (i.e., $C = \int e^{-\lambda h(\bm{w})} d\bm{w}$).
We assume the following two conditions on the function $h$:
\begin{enumerate}
    \item $\forall a \in \mathbb{R},\, \exists \kappa\in\setR_{>0},\, h(a \bm{w}) = |a|^\kappa h(\bm{w})$, 
    \item $\int e^{-h(\bm{w})} d\bm{w} < \infty$.
\end{enumerate}
The first condition is termed absolutely homogeneous of degree $\kappa$, and the second is necessary for the convergence of $C$.
We refer to $h$, possessing the aforementioned properties, as a homogeneous regularization function.
The class of homogeneous regularization functions includes
\begin{itemize}
    \item $h(\bm{w}) = \frac{1}{p}\|\bm{w}\|_p^p$ ($L_p$ regularization, bridge regularization~\cite{frank1993statistical}),
    \item $h(\bm{w}) = \|\bm{w}\|_2$ (group lasso regularization~\cite{yuan2006model} without overlap)
\footnote{
  Empirical Bayesian estimation of $\lambda$ for the group lasso regularizer $\sum_{g\in\mathfrak{G}} \lambda_g \|\bm{w}_g\|_2$ simplifies to analyzing the regularizer $\|\bm{w}\|_2$ in cases without group overlap. This is explained in detail in Prop. \ref{prop: Z_GL Reduction of loss function}. 
},
\end{itemize}
where $\|\cdot\|_p$ is $L_p$ norm and $p \in \mathbb{R}_{\geq 1}$.

The posterior distribution $p(\bm{w} | y^n, \bm{x}^n, \lambda)$ is given by Bayes' theorem:
\begin{align}
    p(\bm{w} | y^n, \bm{x}^n, \lambda) &= \frac{p(y^n|\bm{x}^n, \bm{w}) p(\bm{w} | \lambda)}{Z(\lambda)}, \\ 
    Z(\lambda) &= \int p(y^n|\bm{x}^n, \bm{w}) p(\bm{w} | \lambda)d\bm{w}, \label{eq: marginal likelihood}
\end{align}
where the likelihood is given by $p(y^n|\bm{x}^n, \bm{w}) = \prod_{i=1}^n p(y_i|\bm{x}_i, \bm{w})$ under the i.i.d assumption. 
$Z$ is the marginal likelihood and is crucial in empirical Bayes estimation.
The MAP estimation for the model \eqref{eq: model} reduces to the linear regression problem with a loss function as follows:
\begin{align}
    L(\bm{w}) = \frac{1}{2}\|\bm{y} - \Phi \bm{w}\|_2^2 + \lambda h(\bm{w}), \label{eq: linear regression loss function}
\end{align}
where $\bm{y} = (y_1, \ldots, y_n)^\tp$, and $[\Phi]_{ij} = \phi_{j}(\bm{x}_i)\ (i=1,\ldots,n; j=1,\ldots,m)$. $\Phi \in \setR^{n\times m}$ is known as the design matrix. $[\cdot]_{ij}$ denotes the $(i, j)$th entry of the matrix. 

From the correspondence between the model \eqref{eq: model} and the loss function \eqref{eq: linear regression loss function}, this paper refers to models \eqref{eq: model} with $h(s)=\frac{1}{2}\|\bm{w}\|_2^2, \|\bm{w}\|_1$ and $ \|\bm{w}\|_2$ as ridge, lasso and group lasso model, respectively.
Additionally, we denote $Z$ for ridge, lasso and group lasso as $Z_\mathrm{ridge}, Z_\mathrm{lasso}$ and $Z_\mathrm{GL}$, respectively.


\subsection{Empirical Bayes Estimation}
Empirical Bayes estimation is a method to estimate the hyperparameters.
The estimator of the regularization parameter $\lambda$ is defined by the maximizer of the $Z(\lambda)$:
\begin{align}
    \hat{\lambda} := \argmax_{\lambda} Z(\lambda).
\end{align}
For example, the estimator $\hat{\lambda}_\mathrm{ridge}$ for the ridge model under $\Phi = I$ is given by \cite{nakajima2004Generalization} as follows:
\vspace{-0.25em}
\begin{align}
    \hat{\lambda}_\mathrm{ridge} &= 
        \left\{\begin{array}{cl} 
            \displaystyle \infty & \text{(if $\|\bm{y}\|_2 \leq \sqrt{m}$)}, \\
            \displaystyle \frac{m}{\|\bm{y}\|_2^2 \blue{- m}} & \text{(if $\|\bm{y}\|_2 > \sqrt{m}$)}.
        \end{array}\right. \label{eq: empirical Bayes estimator for ridge}
\end{align}
Subsequently, \cite{yoshida2018empirical} gave $\hat{\lambda}_\mathrm{lasso}$ for the lasso model under $\Phi = I, m = 1$: 
\vspace{-0.25em}
\begin{align}
    \hat{\lambda}_\mathrm{lasso} &= 
        \left\{\begin{array}{cl} 
            \displaystyle \infty & \text{(if $|y| \leq 1$)}, \\
            \displaystyle \lambda^*_\mathrm{lasso} & \text{(if $|y| > 1$)},
        \end{array}\right. \label{eq: empirical Bayes estimator for lasso}
\end{align}
where $\lambda^*_\mathrm{lasso}$ is the unique $\lambda$ satisfying 
\begin{align}
      \frac{dZ_\mathrm{lasso}}{d\lambda}(\lambda^*_\mathrm{lasso}) = 0, \quad 0 < \lambda^*_\mathrm{lasso} < \frac{2}{\sqrt{|y|^2-1}}.
\end{align}
The estimators in Eqs. \eqref{eq: empirical Bayes estimator for ridge} and \eqref{eq: empirical Bayes estimator for lasso} share a common structure, diverging to infinity under certain conditions and converging to finite values otherwise.
This structure induces ARD in MAP estimation, resulting in a sparse MAP solution for $\bm{w}$ (i.e., $\hat{\bm{w}}=\bm{0}$) if $\hat{\lambda} = \infty$ is adopted.
In the derivation of Eqs. \eqref{eq: empirical Bayes estimator for ridge} and \eqref{eq: empirical Bayes estimator for lasso}, it is proven that $Z(\lambda)$ is a quasiconcave function.


\section{Main Theorems}
This section states the two main theorems of this paper and provides their proofs. 
Section \ref{sec: EB group lasso model} and Section \ref{sec: EB regularized linear regression model} present the empirical Bayes estimators for $\lambda$ in the group lasso model and the general regularized model, respectively.

\subsection{Empirical Bayes Estimator for Group Lasso Model} 
\label{sec: EB group lasso model}

In the analysis of group lasso, we assume that $\Phi$ in the model \eqref{eq: model} satisfies the whitening condition for the feature vectors, given by:
\vspace{-0.5em}
\begin{align}
  \Phi^\tp \Phi = nI, \label{eq: whitening condition}
\end{align}
where $I$ is the $m \times m$ identity matrix
\footnote{Even for random data, this condition can be easily satisfied by performing a whitening transformation, which ensures that the data is both uncorrelated and standardized.}.

\begin{proposition}
  \label{prop: Z_GL Reduction of loss function}
  Consider a family of groups $\mathfrak{G} = \{g_1, g_2, \ldots, g_{|\mathfrak{G}|}\}$ and the loss function with hyperparameters $\lambda_g\ (g\in\mathfrak{G})$:
  \begin{align}
      \tilde{L}(\bm{w}) = \frac{1}{2}\|\bm{y} - \Phi \bm{w}\|_2^2 + \sum_{g\in \mathfrak{G}} \lambda_g \|\bm{w}_g\|_2,
  \end{align}
  where $\bm{w}_g$ is the subvector that consists of $w_i\ (i\in g)$, and $i$ represents the indices belonging to group $g$.  
  Under the assumptions of no group overlap (i.e., when $\mathfrak{G}$ is a partition of $\{1,2,\ldots, m\}$) and the whitening condition \eqref{eq: whitening condition},
  empirical Bayesian estimation of $\lambda$ for $\tilde{L}(\bm{w})$ simplifies to analyzing the loss function $L(\bm{w}) = \frac{1}{2}\|\bm{y} - \Phi \bm{w}\|_2^2 + \lambda \|\bm{w}\|_2$.
\end{proposition}
\begin{proof}
  The term $\frac{1}{2}\|\bm{y} - \Phi \bm{w}\|_2^2$ within $\tilde{L}(\bm{w})$ can be rewritten by completing the square and decomposed into a sum over each $g$.
  This operation derives $L(\bm{w})$, allowing the focus to be placed solely on $L(\bm{w})$ as a result.
  The complete proof of this proposition is given in Appendix \ref{sec: Proof of Prop. Z_GL Reduction of loss function}.
\end{proof}
Proposition \ref{prop: Z_GL Reduction of loss function} indicates that the analysis of the model \eqref{eq: model} with $h(\bm{w}) = \|\bm{w}\|_2$ encompasses the analysis of empirical Bayesian estimation for the model with the loss function $\tilde{L}(\bm{w})$.
Henceforth, this paper will focus exclusively on analyzing model \eqref{eq: 
 model} with $h(\bm{w}) = \|\bm{w}\|_2$.

\begin{theorem} \label{theorem: EB group lasso model}
    The empirical Bayes estimator $\hat{\lambda}_\mathrm{GL}$ for the model \eqref{eq: model} under whitening condition \eqref{eq: whitening condition} and $h(\bm{x}) = \|\bm{x}\|_2$ with $m=1,3$ is unique and is given by
    \begin{align}
      \hat{\lambda}_\mathrm{GL} =
        \left\{\begin{array}{cl}
          \infty                & \left(\mathrm{if}\ \|\tilde{\bm{y}}\|_2 \leq \sqrt{m}\right),\\
          \lambda^*_\mathrm{GL} & \left(\mathrm{if}\ \|\tilde{\bm{y}}\|_2 >    \sqrt{m}\right),
        \end{array}\right. \label{eq: EB group lasso model}
    \end{align}
    where 
    $\tilde{\bm{y}} := \frac{1}{\sqrt{n}}\Phi^\tp \bm{y}$, 
    and $\lambda^*_\mathrm{GL}$ is the unique $\lambda$ satisfying
    \begin{align}
       \frac{dZ_\mathrm{GL}}{d\lambda}(\lambda^*_\mathrm{GL}) = 0, \quad 0 < \lambda^*_\mathrm{GL} < \sqrt{\frac{(m+3)mn}{\|\tilde{\bm{y}}\|_2^2-m}}.
    \end{align}
    Furthermore, $\hat{\lambda}_\mathrm{GL}$ is a function that depends only on $\|\tilde{\bm{y}}\|_2$.
\end{theorem}

The estimator at $m=1$ and $n=1$ reduces to Eq. \eqref{eq: empirical Bayes estimator for lasso} since the group lasso model is equivalent to the lasso model in this scenario.
This theorem extends the result of Eqs. \eqref{eq: empirical Bayes estimator for lasso}, and demonstrates that it shares the same structure as Eqs. \eqref{eq: empirical Bayes estimator for ridge} and \eqref{eq: empirical Bayes estimator for lasso}.
In other words, the MAP solution of $\bm{w}$ is sparsified (i.e., $\hat{\bm{w}}=\bm{0}$) by ARD if $\hat{\lambda} = \infty$ is adopted.
The derivation of the theorem proves that $Z_\mathrm{GL}(\lambda)$ is a quasiconcave function, as in the derivation of Eqs. \eqref{eq: empirical Bayes estimator for ridge} and \eqref{eq: empirical Bayes estimator for lasso}.


The proof is derived directly from Prop. \ref{prop: Z_GL quasiconcavity} stated later.
Propositions \ref{prop: Z_GL Reduction of Matrices} and \ref{prop: Z_GL} provide a representation of $Z_\mathrm{GL}$, while Prop. \ref{prop: Z_GL quasiconcavity} elucidates the properties of $Z_\mathrm{GL}$.
Proposition 2 demonstrates that, under the whitening condition, the analysis of $Z_\mathrm{GL}$ is sufficient only for the case where it is the identity matrix. 
Proposition 3 shows that by limiting $m$, $Z_\mathrm{GL}$ can be explicitly expressed as a special function, and its quasiconcavity properties are presented in Prop. 4.

\begin{proposition}
  \label{prop: Z_GL Reduction of Matrices}
  If $\Phi$ satisfies the condition \eqref{eq: whitening condition},
  the marginal likelihood $Z_\mathrm{GL}$ for the model \eqref{eq: model} under $h(\bm{x}) = \|\bm{x}\|_2$ is
  expressed as
  \begin{align}
    Z_\mathrm{GL}(\lambda) 
    \propto \lambda^m \int_{\mathbb{R}^m}  e^{-\frac{1}{2}\|\tilde{\bm{y}} - \bm{w}\|_2^2 - \frac{\lambda}{\sqrt{n}}\|\bm{w}\|_2} d\bm{w}. \label{eq: Z_GL Reduction of Matrices}
  \end{align}
\end{proposition}
\begin{proof}
  This proposition can be demonstrated by sequentially applying completing the square, leveraging the assumption \eqref{eq: whitening condition}, and performing the variable transformation $\tilde{\bm{w}} = \sqrt{n}\bm{w}$ to $Z$.
  The complete proof of this proposition is given in Appendix \ref{sec: Proof of Prop. Z_GL Reduction of Matrices}.
\end{proof}

\begin{proposition}
  \label{prop: Z_GL}
  If $\Phi$ satisfies the condition \eqref{eq: whitening condition} with $m=3$, 
  the marginal likelihood $Z_\mathrm{GL}$ for the model \eqref{eq: model} under $h(\bm{x}) = \|\bm{x}\|_2$ is
  expressed as 
  \begin{align}
    Z_{\mathrm{GL}}(\lambda) 
    \propto \begin{cases}
        \Lambda^3 \left((2\Lambda^2+1)\erfcx(\Lambda)-\frac{2}{\sqrt{\pi}}\Lambda\right)
        & (\mathrm{if}\ \tilde{\bm{y}} = \bm{0}),\\
        \Lambda^3 \left((\Lambda+M)\erfcx(\Lambda+M) \right. \\ \left. \hspace{2em} -(\Lambda-M)\erfcx(\Lambda-M)\right)
        & (\mathrm{if}\ \tilde{\bm{y}} \neq \bm{0}),
      \end{cases} \label{eq: Z_GL}
  \end{align}
  where
  $\erfcx x := \frac{2}{\sqrt{\pi}} e^{x^2} \int_x^\infty e^{-t^2} dt,\ 
  \Lambda := \frac{\lambda}{\sqrt{2n}},\ 
  M := \frac{\|\tilde{\bm{y}}\|_2}{\sqrt{2}}$.
\end{proposition}

\begin{proof}
  When $\bm{y} = \bm{0}$, this can be easily proven using integration by parts. 
  When $\bm{y} \neq \bm{0}$, the integration over $Z$ in Eq. \eqref{eq: Z_GL Reduction of Matrices} can be simplified by first applying a variable transformation $\tilde{\bm{w}} = Q\bm{w}$, 
  where $Q$ is an orthogonal matrix and first row of $Q$ is $\tilde{\bm{y}}/\|\tilde{\bm{y}}\|_2$,
  followed by a transformation to polar coordinates.
  After this operation, performing partial integration yields an expression involving $\erfcx(x)$.
  The complete proof of this proposition is given in Appendix \ref{sec: Proof of Prop. Z_GL}.
\end{proof}


\begin{proposition}
  \label{prop: Z_GL quasiconcavity}
  If $\Phi$ satisfies the condition \eqref{eq: whitening condition} with $m=1,3$,
  $Z_{\mathrm{GL}}(\lambda)$ are quasiconcave functions, exhibiting monotonically increasing behavior for $\|\tilde{\bm{y}}\|_2 \leq \sqrt{m}$, and is unimodal with a maximum at $\lambda^* < \sqrt{\frac{(m+3)mn}{\|\tilde{\bm{y}}\|_2^2-m}}$ for $\|\tilde{\bm{y}}\|_2 > \sqrt{m}$.
\end{proposition}

\begin{proof}
  The proof of this proposition is given in Appendix \ref{sec: Proof of Prop. Z_GL quasiconcavity}.
  In this proof, the function's derivative is computed and decomposed into a product of multiple functions. This process is repeated until the properties of the derivative can be analytically determined. This proof can be applied to marginal likelihoods that involve functions where the integration range includes $x$, such as $\erfcx(x)$.
\end{proof}


\subsection{Empirical Bayes Estimator for General Regularization Model} \label{sec: EB regularized linear regression model}
\begin{theorem}\label{theorem: EB regularized linear regression model}
    If the model \eqref{eq: model} satisfies
    \begin{align}
      &E_{p(\bm{w}|\lambda)}[\bm{w}] = \bm{0},
      &E_{p(\bm{w}|\lambda)}[\bm{w}\bm{w}^\mathrm{T}] = \sigma_{\bm{w}}^2 I, \label{eq: condition of asymptotic expansion E[w]=0, E[ww]=sigma^2I}
    \end{align}
    and the marginal likelihood $Z(\lambda)$ is a quasiconcave function,  the empirical Bayes estimate $\hat{\lambda}$ for the model \eqref{eq: model} is unique and is given by
    \begin{align}
      \hat{\lambda} =
        \left\{\begin{array}{cl}
          \infty    & \left(\mathrm{if}\ \|\Phi^\mathrm{T}\bm{y}\|_2 \leq \|\Phi\|_F\right),\\
          \lambda^* & \left(\mathrm{if}\ \|\Phi^\mathrm{T}\bm{y}\|_2 >  \|\Phi\|_F\right),
        \end{array}\right. \label{eq: empirical Bayes estimator for General Regularization Model}
    \end{align}
    where $\sigma_{\bm{w}}^2 \in \mathbb{R}_{\geq 0}$ is a constant, $\|\cdot\|_F$ is the Frobenius norm, and $\lambda^*$ is the unique $\lambda$ satisfying $\frac{dZ}{d\lambda}(\lambda^*) = 0$.
\end{theorem}

Theorem \ref{theorem: EB regularized linear regression model} is a generalization
of Eqs. \eqref{eq: empirical Bayes estimator for ridge} and \eqref{eq: empirical Bayes estimator for lasso}, 
as well as Thm. \ref{theorem: EB group lasso model}, 
detailing the conditions under which models produce the ARD mechanism.
More specifically, only by checking the condition, $\|\Phi^\mathrm{T}\bm{y}\|_2 \leq \|\Phi\|_F$, we know before starting any estimation procedure if the empirical Bayes estimator diverges and $\bm{w}$ is estimated to be zero.


Theorem \ref{theorem: EB regularized linear regression model} is directly proven from Prop. \ref{prop: Z asymptotic expansion} and the quasiconcavity of $Z(\lambda)$.


\begin{proposition}
  \label{prop: Z asymptotic expansion}
  If the model \eqref{eq: model} satisfies
  the condition \eqref{eq: condition of asymptotic expansion E[w]=0, E[ww]=sigma^2I},
  the marginal likelihood $Z(\lambda)$ has an asymptotic expansion
  \begin{align*}
    Z(\lambda)
    &\!=\! \frac{e^{-\frac{1}{2}\|\bm{y}\|_2^2}}{(2\pi)^\frac{n}{2}} 
       \!\left(\! 1 \!+\! \frac{\sigma_{\bm{w}}^2}{2} \!\left(\|\Phi^\mathrm{T}\bm{y}\|_2^2 \!-\! \|\Phi\|_F^2\right)\! \lambda^{-\frac{2}{\kappa}} \!\right)
       \!+ O \! \left(\lambda^{-\frac{3}{\kappa}}\!\right)\!
  \end{align*}
  \blue{as $\lambda \to \infty$.}
\end{proposition}

\begin{proof}
  The variable transformation $\tilde{\bm{w}} = \lambda^{\frac{1}{\kappa}}\bm{w}$ absorbs all instances of $\lambda$ outside the integral into the exponential function within the integral.
  Next, we perform the series expansion of the exponential function. 
  The assumption \eqref{eq: condition of asymptotic expansion E[w]=0, E[ww]=sigma^2I} eliminates the terms involving $\lambda^{-\frac{1}{\kappa}}$.
  Finally, the application of the trace property $\bm{x}^\tp\bm{y} = \tr(\bm{x}\bm{y}^\tp)$ and the linearity of the trace operator to the terms involving $\lambda^{-\frac{2}{\kappa}}$ leads to the derivation of the proposition.
  The complete proof of this proposition is given in Appendix \ref{sec: Proof of Prop. Z asymptotic expansion}.
\end{proof}


Do the ridge, lasso, and group lasso models satisfy the condition \eqref{eq: condition of asymptotic expansion E[w]=0, E[ww]=sigma^2I}?
The following Lemma \ref{lemma: p(w) moment} conveniently serves to check whether the model \eqref{eq: model} satisfies the condition \eqref{eq: condition of asymptotic expansion E[w]=0, E[ww]=sigma^2I}.
According to this lemma, the ridge, lasso, and group lasso models are all asymptotically expandable, as in Prop. \ref{prop: Z asymptotic expansion}.
\begin{lemma} \label{lemma: p(w) moment}
  Let $\mathcal{S}$ denote the set of diagonal matrices in $\setR^{m\times m}$, whose diagonal elements are either $-1$ or $+1$, and let $\mathcal{P}$ be the set of \blue{permutation matrices} in $\setR^{m\times m}$.
  Assume that $h$, as specified in model \eqref{eq: model}, fulfills the conditions:
    \begin{enumerate}
      \item For all $\bm{x} \in \setR^m, S \in \mathcal{S}$, $h(S\bm{x}) = h(\bm{x})$: 
      \blue{$h$ is invariant to elementwise sign reversal of $\bm{x}$.}
      \item For all $\bm{x} \in \setR^m, P_\pi \in \mathcal{P}, h(P_\pi\bm{x}) = h(\bm{x})$: \blue{$h$ is invariant to permutations of the components of $\bm{x}$.}
    \end{enumerate}
  Under these assumptions, Eq. \eqref{eq: condition of asymptotic expansion E[w]=0, E[ww]=sigma^2I} is satisfied.
\end{lemma}

\begin{proof}
  Using the symmetry property $h(-\bm{w}) = h(\bm{w})$ and a change of integration variables, 
  we obtain $E_{p(\bm{w}|\lambda)}[\bm{w}] = -E_{p(\bm{w}|\lambda)}[\bm{w}]$. 
  This implies that $E_{p(\bm{w}|\lambda)}[\bm{w}] = 0$.
  Similarly, by using the assumptions of invariance to sign reversal and ingredient replacement, 
  we derive $E_{p(\bm{w}|\lambda)}[\bm{w}\bm{w}^\tp] = S_{(i)} E_{p(\bm{w}|\lambda)}[\bm{w}\bm{w}^\tp] S_{(i)}$ 
  and $E_{p(\bm{w}|\lambda)}[\bm{w}\bm{w}^\tp] = P_{(i,j)} E_{p(\bm{w}|\lambda)}[\bm{w}\bm{w}^\tp] P_{(i,j)}$, respectively, where $S_{(i)}$ is a diagonal matrix with the $i$-th diagonal element equal to $-1$ and all other diagonal elements being $+1$ and $P_{(i,j)}$ is the matrix obtained by swapping the $i$-th and $j$-th columns of the identity matrix.
  These results imply that $E_{p(\bm{w}|\lambda)}[\bm{w}\bm{w}^\tp] = \sigma_{\bm{w}}^2 I$.
  The complete proof of this lemma is given in Appendix \ref{sec: Proof of Lemma p(w) moment}.
\end{proof}

\section{Discussion}
\subsection{Quasiconcavity of $Z(\lambda)$}
The conditions for the empirical Bayes estimator to be expressed in the form of Thm. \ref{theorem: EB regularized linear regression model} include the quasiconcavity of $Z(\lambda)$.
The derivations of Eqs. \eqref{eq: empirical Bayes estimator for ridge} and \eqref{eq: empirical Bayes estimator for lasso}, as well as Thm. \ref{theorem: EB group lasso model}, individually demonstrate the quasiconcavity.
What types of models lead to $Z(\lambda)$ being quasiconcave?
The pattern for which quasiconcavity has been rigorously proven suggests that $Z(\lambda)$ is quasiconcave for a general $\Phi$.

Here, we approximate $Z(\lambda)$ using Monte Carlo integration to observe the rough shape of its graph. 
The Monte Carlo integration is performed using
\begin{align}
  Z(\lambda)
  &= (2\pi)^{-\frac{n}{2}} \int_{\mathbb{R}^m} p_0(\bm{w}) \exp\left(-\frac{1}{2}\|\bm{y}-\lambda^{-\frac{1}{\kappa}}\Phi\bm{w}\|_2^2 \right) d\bm{w} \nonumber\\
  & \approx (2\pi)^{-\frac{n}{2}} \sum_{k=1}^N \exp\left(-\frac{1}{2}\|\bm{y}-\lambda^{-\frac{1}{\kappa}}\Phi\bm{w}_k\|_2^2 \right), \label{eq: monte carlo integration}
\end{align}
where $p_0(\bm{w}) = p(\bm{w}|\lambda)|_{\lambda=1}$, and $\bm{w}_k\ (k=1,\ldots,N)$ are samples from $p_0(\bm{w})$.
\footnote{The first equality is obtained during the proof of Prop. \ref{prop: Z asymptotic expansion} in Appendix \ref{sec: Proof of Prop. Z asymptotic expansion}.}

We calculate the $Z(\lambda)$ values for ridge, lasso, and group lasso.
Sampling of $\bm{w}_k$ from $p_0$ in ridge and lasso regression is performed using a standard Gaussian distribution and a standard Laplace distribution, respectively.
In group lasso, $\bm{w}_k$ is computed from $r_k\in\setR_{\geq0}$ and $\bm{x}_k\in\setR^m$ as $\bm{w}_k = r_k\cdot\frac{\bm{x}_k}{\|\bm{x}_k\|_2}$, where $r_k$ is sampled from a Gamma distribution of order $m$, and $\bm{x}_k$ is drawn from a standard Gaussian distribution~\cite{Boisbunon2012TheCO}.
The constant coefficient $(2\pi)^{-\frac{n}{2}}$ in Eq. \eqref{eq: monte carlo integration} is omitted from the calculation to prevent computational underflow. 
The vector $\bm{y}$ and the matrix $\Phi$ are generated as random variables following the standard Gaussian distribution.

\begin{figure}[tbp]
  \centering
  \includegraphics[width=0.5\textwidth]{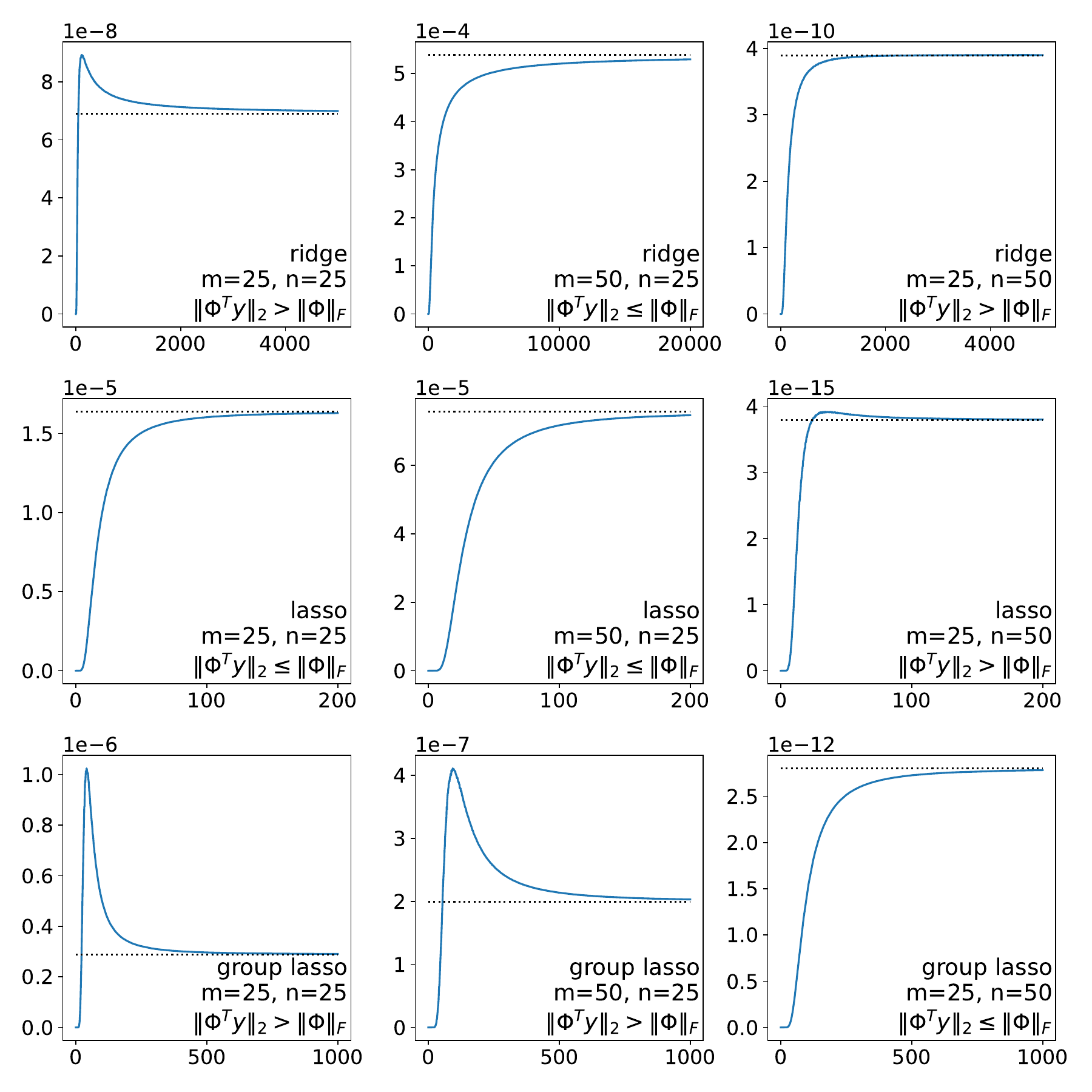}
  \caption{
  Results of the Monte Carlo integration for $Z(\lambda)$. 
  The solid lines in the figure represent $Z(\lambda)$, while the dotted lines represent the asymptotes.
  The upper row corresponds to ridge, the middle row to lasso, and the bottom row to group lasso.
  The bottom-right corner of each graph displays the respective settings of $m$ and $n$, 
  along with the evaluation results of the branch condition in Eq. \eqref{eq: empirical Bayes estimator for General Regularization Model}.
  All the calculation results demonstrate that $Z(\lambda)$ exhibits quasiconcavity, and the branch condition results govern the function's shape.
  }
  \label{fig: monte carlo integration}
\end{figure}

Figure \ref{fig: monte carlo integration} shows the results of $Z(\lambda)$ for various patterns of $m$ and $n$ with $N=5\times10^6$.
Numerical calculations are performed for each model under the conditions $m = n$, $m > n$, and $m < n$, respectively.
In all cases, the shape of the graph indicates that $Z(\lambda)$ exhibits quasiconcavity.
The dotted lines in the figure represent the asymptotes derived from Prop. \ref{prop: Z asymptotic expansion}.
According to the evaluation results of the branch condition in Eq. \eqref{eq: empirical Bayes estimator for General Regularization Model}, 
$Z(\lambda)$ is either a monotonically increasing function that approaches the asymptote from below or a unimodal function that approaches the asymptote from above.
From the results of this numerical computation, we propose the following conjecture.
\begin{conjecture}
  \label{conj: Z quasiconcavity}
  The marginal likelihood $Z(\lambda)$ for ridge, lasso, and group lasso is a quasiconcave function.
\end{conjecture}
In other words, the empirical Bayes estimators for these models are conjectured to be expressed in the form presented in Thm. \ref{theorem: EB regularized linear regression model}.

\subsection{On the Application of Theorem \ref{theorem: EB regularized linear regression model}}
Our results, Thm. \ref{theorem: EB regularized linear regression model} and Conj. \ref{conj: Z quasiconcavity}, suggest that the empirical Bayes estimator of $\lambda$ may diverge in various models. Theorem \ref{theorem: EB regularized linear regression model} indicates that this divergence occurs automatically when the condition $\|\Phi^\tp\bm{y}\|_2 \leq \|\Phi\|_F$ is satisfied.
For calculating the empirical Bayes estimator, a typical approach involves estimating $\lambda$ using gradient methods on the derivative of the marginal likelihood. However, by applying Thm. 2, when $\|\Phi^\tp\bm{y}\|_2 \leq \|\Phi\|_F$ holds, $\lambda$ can be estimated as $\lambda = \infty$ without performing any gradient computation.
This leads to $\hat{\bm{w}} = \bm{0}$, which in turn allows the MAP estimation of $\bm{w}$ to be performed without any additional computations.
When conducting empirical Bayes estimation for more complex prior distributions than those considered in this study, a local variational approximation that approximates the prior distribution with a Gaussian distribution could be employed. Even in such cases, this theorem remains useful.

\section{Conclusion}
We rigorously derived the empirical Bayes estimators for the group lasso, 
which diverge under specific conditions, thereby leading to the ARD mechanism.
Although we restricted ourselves to the simple models, $m=1$ and $m=3$, to prove that the marginal likelihood is quasiconcave (unimodal) and that ARD occurs, these properties are expected to hold for more general models as well.
We also proved that the empirical Bayes method induces the ARD mechanism on the general regularized linear regression models. 
This elucidates the specific conditions under which models, including ridge, lasso, and group lasso, activate the ARD mechanism.
Our future directions include characterizing the class of regularizers leading to the quasiconcavity of the marginal likelihood.


%
\bibliographystyle{IEEEtran}
\bibliography{isit2025}
%


\newpage 
\appendix
\appendices


\subsection{Proof of Prop. \ref{prop: Z_GL Reduction of loss function}} \label{sec: Proof of Prop. Z_GL Reduction of loss function}
Before proceeding with the proof, we define the following variables.
Let $\bm{u} \coloneqq (\Phi^\tp\Phi)^{-1} \Phi^\tp \bm{y}$, 
and let $\bm{u}_g$ denotes the subvector that consists of $u_i\ (i\in g)$.
The matrix $\Phi_g \in \setR^{n \times m_g}$ is defined such that $\Phi_g^\tp \Phi_g = nI_{m_g}$, where $l$ is an arbitrary positive integer and $I_{m_g}$ is the $m_g\times m_g$ identity matrix.
Finally, let $\bm{y}_g \coloneqq \Phi_g \bm{u}_g$. 
  
The loss function $\tilde{L}(\bm{w})$ can be transformed as follows through the method of completing the square, under the whitening condition \eqref{eq: whitening condition}:
\begin{align*}
  \tilde{L}(\bm{w})
  &= \frac{1}{2}(\bm{w}-\bm{u})\Phi^\tp \Phi(\bm{w}-\bm{u}) + \sum_{g\in\mathfrak{G}} \lambda_g \|\bm{w}_g\|_2 + C \\
  &= \sum_{g\in\mathfrak{G}} \left(\frac{n}{2} \sum_{i \in g} (w_i-u_i)^2 + \lambda_g \|\bm{w}_g\|_2\right) + C \\
  &= \sum_{g\in\mathfrak{G}} \left(\frac{1}{2}\|\bm{y}_g-\Phi_g\bm{w}_g\|_2^2 + \lambda_g \|\bm{w}_g\|_2\right) + C \\
  &= \sum_{g\in\mathfrak{G}} L_g(\bm{w}_g) + C
\end{align*}
where $L_g(w_g):=\frac{1}{2}\|\bm{y}_g-\Phi_g\bm{w}_g\|_2^2 + \lambda_g \|\bm{w}_g\|_2$, and $C$ is a constant that does not depend on $\bm{w}$.

Consequently, the empirical Bayes estimation of $\hat{\lambda}_{g_k}$ reduces to the problem of empirical Bayes estimation using only $L_{g_k}(\bm{w}_{g_k})$:

\begin{align*}
  \hat{\lambda}_{g_k}
  &= \argmax_{\lambda_{g_k}}\ Z\left(\lambda_{g_1}, \ldots, \lambda_{g_{|\mathfrak{G}|}}\right)\\
  &= \argmax_{\lambda_{g_k}} \int e^{-L(\bm{w})} d\bm{w}\\
  &= \argmax_{\lambda_{g_k}} \int e^{-L_{g_k}\left(\bm{w}_{g_k}\right)} d\bm{w}_{g_k} \!\cdot\! \int e^{-\sum_{g\in\mathfrak{G}\setminus{g_k}}\!\! L_g\left(\bm{w}_g\right)} d\bm{w}_g\\
  &= \argmax_{\lambda_{g_k}} \int e^{-L_{g_k}\left(\bm{w}_{g_k}\right)} d\bm{w}_{g_k}.
\end{align*}
In other words, this reduces to the problem of a loss function with a single $\lambda$: $L(\bm{w}) = \frac{1}{2}\|\bm{y} - \Phi \bm{w}\|_2^2 + \lambda \|\bm{w}\|_2$. 


\subsection{Proof of Prop. \ref{prop: Z_GL Reduction of Matrices}} \label{sec: Proof of Prop. Z_GL Reduction of Matrices}
The marginal likelihood can be reformulated as follows by sequentially completing the square, leveraging assumption \eqref{eq: whitening condition}, and performing the variable transformation $\tilde{\bm{w}} = \sqrt{n}\bm{w}$:
\begin{align*}
  Z(\lambda)
  &\propto \lambda^m \!\int_{\mathbb{R}^m}\! \exp\left(-\frac{1}{2}\|\bm{y}-\Phi\bm{w}\|_2^2 -\lambda \|\bm{w}\|_2\right) d\bm{w}\\
  &\propto \lambda^m \!\int_{\mathbb{R}^m}\! \exp\left(\!-\frac{1}{2}(\bm{w}\!-\!\bm{y}')\Phi^\tp\Phi(\bm{w}\!-\!\bm{y}') \!-\!\lambda \|\bm{w}\|_2\!\right) d\bm{w}\\
  &\propto \lambda^m \!\int_{\mathbb{R}^m}\! \exp\left(-\frac{n}{2}\|\bm{w}-\bm{y}'\|_2^2 -\lambda\|\bm{w}\|_2\right) d\bm{w}\\
  &\propto \lambda^m \!\int_{\mathbb{R}^m}\! \exp\left(-\frac{1}{2}\|\bm{w}-\sqrt{n}\bm{y}'\|_2^2 -\frac{\lambda}{\sqrt{n}} \|\bm{w}\|_2\right) d\bm{w},
\end{align*}
where $\bm{y}' := (\Phi^\tp\Phi)^{-1}\Phi^\tp\bm{y} = \frac{1}{n}\Phi^\tp\bm{y}$.


\subsection{Proof of Prop. \ref{prop: Z_GL}} \label{sec: Proof of Prop. Z_GL}
Substituting $m=3$ into Eq. \eqref{eq: Z_GL Reduction of Matrices} in Prop. \ref{prop: Z_GL Reduction of Matrices} yields
\begin{align*}
  Z_\mathrm{GL}(\lambda)
     \propto \lambda^3 
      \int_{\mathbb{R}^m} e^{-\frac{1}{2}\|\bm{w}\|_2^2 -\lambda\|\bm{w}\|_2 + \tilde{\bm{y}}^\mathrm{T}\bm{w}} d\bm{w}.
\end{align*}

If $\tilde{\bm{y}} = \bm{0}$, a polar coordinate transformation transforms the integral into
\begin{align*}
  \int_{\mathbb{R}^m} e^{-\frac{1}{2}\|\bm{w}\|_2^2 -\lambda\|\bm{w}\|_2 + \tilde{\bm{y}}^\mathrm{T}\bm{w}} d\bm{w}
  = 4\pi e^\frac{\lambda^2}{2}\int_0^\infty r^2e^{-\frac{1}{2}\left(r+\lambda\right)^2}dr.
\end{align*}
Then, by applying integration by parts, we obtain the result stated in Prop. \ref{prop: Z_GL}.

If $\tilde{\bm{y}} \neq \bm{0}$, we perform a variable transformation for the integral part, given by $\tilde{\bm{w}}=Q\bm{w}$, where $Q$ is an orthogonal matrix and first row of $Q$ is $\tilde{\bm{y}}/\|\tilde{\bm{y}}\|_2$. 
The polar transformation $(w_1', w_2', w_3') = (r\cos\phi, r\sin\phi \cos\theta, r\sin\phi \sin\theta)$ after such a variable transformation leads to
\begin{align*}
  &\int_{\mathbb{R}^m} e^{-\frac{1}{2}\|\bm{w}\|_2^2 -\lambda\|\bm{w}\|_2 + \tilde{\bm{y}}^\mathrm{T}\bm{w}} d\bm{w} \\
  &= 2\pi \int_0^\infty r^2 e^{-\frac{1}{2}r^2 -\lambda r} \left(\int_0^\pi e^{\|\tilde{\bm{y}}\|_2 r\cos\phi}\sin\phi d\phi \right)dr \\
  &= 2\pi \int_0^\infty r^2 e^{-\frac{1}{2}r^2 -\lambda r} \left(\frac{1}{\|\tilde{\bm{y}}\|_2 r}\left(e^{\|\tilde{\bm{y}}\|_2 r} -e^{-\|\tilde{\bm{y}}\|_2 r}\right)\right)dr\\
  &= \frac{2\pi}{\|\tilde{\bm{y}}\|_2}\left(
    e^{\frac{1}{2}\left(\lambda-\|\tilde{\bm{y}}\|_2\right)^2}
    \int_0^\infty re^{-\frac{1}{2}\left(r+\lambda-\|\tilde{\bm{y}}\|_2\right)^2}dr\right.\\
  &\left.\hspace{6em} -e^{\frac{1}{2}\left(\lambda+\|\tilde{\bm{y}}\|_2\right)^2}
    \int_0^\infty re^{-\frac{1}{2}\left(r+\lambda+\|\tilde{\bm{y}}\|_2\right)^2}dr
  \right).
\end{align*}
Finally, by applying integration by parts, we have the result of the Prop. \ref{prop: Z_GL}.


\subsection{Proof of Prop. \ref{prop: Z_GL quasiconcavity}} \label{sec: Proof of Prop. Z_GL quasiconcavity}
Since the case $m=1$ reduces to the discussion in \cite{yoshida2018empirical}, we discuss the case for $m=3$.
We prove $Z_\mathrm{GL}(\lambda)$ in Prop. \ref{prop: Z_GL} is a quasiconcave function on $\lambda>0$ by considering cases where (A) $\tilde{\bm{y}}=\bm{0}$ and (B) $\tilde{\bm{y}}\neq\bm{0}$.

(A) If $\tilde{\bm{y}} = \bm{0}$, 
we define the following functions and compute their derivatives:
\begin{align*}
  f(\Lambda) &\coloneqq \Lambda^3 \left((2\Lambda^2+1)\erfcx(\Lambda)-\frac{2}{\sqrt{\pi}}\Lambda\right),\\
  f'(\Lambda) &= \Lambda^2(4\Lambda^4+12\Lambda^2+3)\erfcx\Lambda-\frac{2}{\sqrt{\pi}}\Lambda^3(2\Lambda^2+5)\\
  &= \frac{2}{\sqrt{\pi}}\Lambda^2(4\Lambda^4+12\Lambda^2+3) u(\Lambda), \\
  u(\Lambda) &\coloneqq \int_\Lambda^\infty e^{-t^2}dt - \frac{\Lambda(2\Lambda^2+5)}{4\Lambda^4+12\Lambda^2+3}e^{-\Lambda^2}\\
  u'(\Lambda) &= -\frac{24}{(4\Lambda^4+12\Lambda^2+3)^2}e^{-\Lambda^2}
\end{align*}
The proof is achieved by verifying that $f$ is a quasiconvex function for $\Lambda > 0$.
Since $u'(\Lambda) < 0$ and $\lim_{\Lambda\to\infty} u(\Lambda) = 0$, $u(\Lambda) > 0$. This means $f'(\Lambda) > 0$, and thus $f(\Lambda)$ is monotonically increasing.
Consequently, $f$ is quasiconcave.

(B) If $\tilde{\bm{y}} \neq \bm{0}$,
by using the five polynomials of $\Lambda$,
\begin{align*}
  P_1(\Lambda) &\coloneqq 2\Lambda^3 +4M\Lambda^2 +2(M^2+2)\Lambda +3M,\\
  P_2(\Lambda) &\coloneqq 2\Lambda^3 -4M\Lambda^2 +2(M^2+2)\Lambda -3M,\\
  P_3(\Lambda) &\coloneqq 4\Lambda^6 -4(2M^2-3)\Lambda^4 \\&\hspace{2em} +(4M^4-4M^2+15)\Lambda^2 -6(2M^2+1),\\
  P_4(\Lambda) &\coloneqq 2\Lambda^5 +2M\Lambda^4 -(2M^2-5)\Lambda^3 \\&\hspace{2em} -(2M^2-3)M\Lambda^2 +6\Lambda +6M,\\
  P_5(\Lambda) &\coloneqq (2M^2-3)\Lambda^2 -9,
\end{align*}
we define the following functions and compute their derivatives:
\begin{align*}
    g(\Lambda) 
       &\coloneqq \textstyle \Lambda^3 (\Lambda+M)\erfcx(\Lambda+M)\\ 
       &\hspace{2em} -\Lambda^3(\Lambda-M)\erfcx(\Lambda-M),\\
    g'(\Lambda) 
       &= \Lambda^2 P_1(\Lambda) \mathrm{erfcx}(\Lambda+M) \\
       &\hspace{2em} +\Lambda^2 P_2(\Lambda) \mathrm{erfcx}(\Lambda-M) - \frac{4}{\sqrt{\pi}}M\Lambda^3\\
       &= \phi_1(\Lambda) \phi_2(\Lambda), \\
    \phi_1(\Lambda) 
       &\coloneqq \frac{2}{\sqrt{\pi}}\Lambda^2 P_1(\Lambda) e^{(\Lambda+M)^2}, \\
    \phi_2(\Lambda)
       &\coloneqq \int_{\Lambda+M}^\infty e^{-t^2}dt
        -\frac{P_2(\Lambda)}{P_1(\Lambda)} e^{-4M\Lambda} \int_{\Lambda-M}^\infty e^{-t^2}dt \\
       &\hspace{2em} -\frac{2M\Lambda}{P_1(\Lambda)} e^{-(\Lambda+M)^2},\\
    \phi_2'(\Lambda)
      &=\frac{4M P_3(\Lambda)}{(P_1(\Lambda))^2} e^{-4M\Lambda} \int_{\Lambda-M}^\infty e^{-t^2}dt\\
      &\hspace{2em} -\frac{4M P_4(\Lambda)}{(P_1(\Lambda))^2} e^{-(\Lambda+M)^2}\\
      &= \psi_1(\Lambda) \psi_2(\Lambda),\\
    \psi_1(\Lambda) &\coloneqq \frac{4M P_3(\Lambda)}{(P_1(\Lambda))^2}e^{-4M\Lambda}, \\
    \psi_2(\Lambda)
      &\coloneqq \int_{\Lambda-M}^\infty e^{-t^2}dt
        -\frac{P_4(\Lambda)}{P_3(\Lambda)}e^{-(\Lambda-M)^2}, \\
    \psi_2'(\Lambda) &= -\frac{8\Lambda P_1(\Lambda) P_5(\Lambda)}{(P_3(\Lambda))^2}e^{-(\Lambda-M)^2}.
\end{align*}
The proof is achieved by verifying that $g$ is a quasiconvex function for $\Lambda > 0$ and $M > 0$.
All functions mentioned, except $\psi_2$ and $\psi_2'$, are continuous for all $\Lambda > 0$ and $M > 0$ since $P_1(\Lambda) > 0$.
$\psi_2$ and $\psi_2'$ have singularities at the zero point of $P_3$.
It can be demonstrated that, under the condition $M>0$, 
$P_3$ consistently possesses a single zero point, $s_{P_3}$, in the range $\Lambda>0$. $s_{P_3}$ satisfies $P_4(s_{P_3}) > 0$
\footnote{
  The existence of $s_{P_3}$ can be established by applying Sturm's theorem.
  To prove that $P_4(s_{P_3}) > 0$, 
  given the transformation $P_4(\Lambda) = \{2\Lambda^4-(2M^2-3)\Lambda^2+6\}(\Lambda+M)+2\Lambda^3$, 
 it suffices to show that $s_{P_3}^2 < s_c = \frac{1}{4}\bigl\{(2M^2-3)-\sqrt{(2M^2-3)-48}\bigr\}$, where $s_c$ is the point at which $2\Lambda^4-(2M^2-3)\Lambda^2+6=0$.
  Considering $P_3(\Lambda)< 0$ for $\Lambda \in (0,\;s_{P_3})$ and $P_3(\Lambda)> 0$ for $\Lambda \in (s_{P_3},\;\infty)$, we only need to confirm that $P_3(s_c) > 0$.
}.
The properties of $P_5$ change at the boundary
$M=\sqrt{\frac{3}{2}}$.
We further categorize our analysis into two cases: (B-1) $M \leq \sqrt{\frac{3}{2}}$ and (B-2) $M > \sqrt{\frac{3}{2}}$.

(B-1) If $M \leq \sqrt{\frac{3}{2}}$, then $\psi_2'>0$, 
given that both $P_1(\Lambda)>0$ and $P_5(\Lambda)<0$.
This fact, combined with $\psi_2(0)>0$ and $\lim_{\Lambda\to\infty}\psi_2=0$ yields 
$\psi_2(\Lambda)>0$ for $\Lambda \in (0,\;s_{P_3})$
and $\psi_2(\Lambda)<0$ for $\Lambda \in (s_{P_3},\;\infty)$.
On the other hand, based on the sign chart of $P_3$, 
it follows that 
$\psi_1(\Lambda)<0$ for $\Lambda \in [0,\;s_{P_3})$,
$\psi_1(\Lambda)>0$ for $\Lambda \in (s_{P_3},\;\infty)$,
and $\psi_1(s_{P_3}) = 0$.
As a result, the sign charts of $\psi_1$ and $\psi_2$, together with $P_4(s_{P_3})>0$, show that $\phi_2'(\Lambda)=\psi_1(\Lambda)\psi_2(\Lambda) < 0$ for all $\Lambda > 0$.
Hence, $\phi_2(\Lambda)>0$ follows from $\lim_{\Lambda\to\infty}\phi_2'(\Lambda)=0$, 
and $g'(\Lambda) = \phi_1(\Lambda)\phi_2(\Lambda) > 0$ is ensured, 
given that $\phi_1(\Lambda) > 0$ for all $\Lambda > 0$.
Therefore, $g$ is monotonically increasing; thus, it is quasiconcave.

(B-2) If $M > \sqrt{\frac{3}{2}}$, then $P_5$ possesses a single zero point, $s_{P_5} = \frac{3}{\sqrt{2M^2-3}}$, within the range $\Lambda>0$.
The calculation of $P_3(s_{P_5})$ indicates that $P_3(s_{P_5})>0$, implying that $s_{P_3} < s_{P_5}$.
The sign charts of $P_3$ and $P_5$ demonstrates that 
$\psi_2'(\Lambda)>0$ for $\Lambda \in (0,\;s_{P_5})$, 
$\psi_2'<0$ for $\Lambda \in (s_{P_5},\;\infty)$,
and $\psi_2'(s_{P_5})=0$.
These facts, combined with $\psi_2(0)>0,\lim_{\Lambda\to\infty} \psi_2(\Lambda) = 0,P_4(s_{P_3})>0$ yields 
$\psi_2(\Lambda)>0$ for $\Lambda \in (0,\;s_{P_3})$, 
$\psi_2(\Lambda)<0$ for $\Lambda \in (s_{P_3},\;s_{\psi_2})$, 
and $\psi_2(\Lambda)>0$ for $\Lambda \in (s_{\psi_2},\;\infty)$, 
where $s_{\psi_2} \in (s_{P_3},\;s_{P_5})$ is a unique point satisfying $\psi_2(s_{\psi_2}) = 0$.
The intermediate value theorem guarantees the existence of $s_{\psi_2}$.
Transitioning to the analysis of $\psi_1$, based on the sign chart of $P_3$, it follows that 
$\psi_1(\Lambda)<0$ for $\Lambda \in (0,\;s_{P_3})$,
$\psi_1(\Lambda)>0$ for $\Lambda \in (s_{P_3},\;\infty)$,
and $\psi_1(s_{P_3})=0$.
The sign chart of $\psi_1$ and $\psi_2$, 
given that $P_4(s_{P_3})>0$ and $\phi_2(\Lambda)=\psi_1(\Lambda)\psi_2(\Lambda)$,
reveals that 
$\phi_2'(\Lambda)<0$ for $\Lambda \in (0,\;s_{\psi_2})$,
$\phi_2'(\Lambda)>0$ for $\Lambda \in (s_{\psi_2},\;\infty)$,
and $\phi_2'(s_{\psi_2})=0$.
Through the sign chart of $\phi_2'$ and $\lim_{\Lambda\to\infty}\phi_2'(\Lambda)=0$,
we establish that 
$\phi_2(\Lambda)>0$ for $\Lambda \in (0,\;s_{\phi_2})$ 
and $\phi_2(\Lambda)<0$ for $\Lambda \in (s_{\phi_2},\;\infty)$, 
where $s_{\phi_2} \in (0,\;s_{\psi_2})$ is a unique point satisfying $\phi_2(s_{\phi_2}) = 0$.
The existence of $s_{\psi_2}$ is guaranteed by the intermediate value theorem, akin to that of $s_{\psi_2}$.
From $\phi_1(\Lambda)>0$ and $g' = \phi_1(\Lambda)\phi_2(\Lambda)$, we see that 
$g'(\Lambda)>0$ for $\Lambda \in (0,\;s_{\phi_2})$,
$g'(\Lambda)<0$ for $\Lambda \in (s_{\phi_2},\;\infty)$,
and $g'(s_{\phi_2})=0$.
Therefore, $g$ is 
strictly monotonically increasing for $\Lambda \in (0,\;s_{\phi_2})$ 
and strictly monotonically decreasing for $\Lambda \in (s_{\phi_2},\;\infty)$, thus, quasiconcave.
According to the proof, $s_{\phi_2}$ uniquely satisfies $g'=0$, and is bounded above by $s_{P_5}$.


\subsection{Proof of Prop. \ref{prop: Z asymptotic expansion}} \label{sec: Proof of Prop. Z asymptotic expansion}
By substituting the model \eqref{eq: model} into the definition of the marginal likelihood \eqref{eq: marginal likelihood} and applying the variable transformation $\tilde{\bm{w}} = \lambda^{\frac{1}{\kappa}}\bm{w}$, we have
\begin{align*}
  Z(\lambda)
  &= \frac{1}{(2\pi)^\frac{n}{2}\int_{\mathbb{R}^m} e^{-h(\bm{w})}d\bm{w}} \\
  &\hspace{2em} \cdot \int_{\mathbb{R}^m} \exp\left(-\frac{1}{2}\|\bm{y}-\lambda^{-\frac{1}{\kappa}}\Phi\bm{w}\|_2^2 -h(\bm{w})\right) d\bm{w}\\
  &= \frac{\exp\left(-\frac{1}{2}\|\bm{y}\|_2^2\right)}{(2\pi)^\frac{n}{2}} \\
  &\hspace{1em} \cdot \int_{\mathbb{R}^m} p_0(\bm{w}) \exp\left(\lambda^{-\frac{1}{\kappa}} \bm{y}^\mathrm{T}\Phi\bm{w}-\frac{1}{2}\lambda^{-\frac{2}{\kappa}}\|\Phi\bm{w}\|_2^2\right) d\bm{w},
\end{align*}
where $p_0(\bm{w}) = p(\bm{w}|\lambda)|_{\lambda=1}$.
The exponential function within this integral can be expanded into a series: 
\begin{align*}
  &e^{\lambda^{-\frac{1}{\kappa}} \bm{y}^\mathrm{T}\Phi\bm{w}-\frac{1}{2}\lambda^{-\frac{2}{\kappa}}\|\Phi\bm{w}\|_2^2}\\
  &= 1 + \lambda^{-\frac{1}{\kappa}} \bm{y}^\mathrm{T}\Phi\bm{w}  
       + \frac{\lambda^{-\frac{2}{\kappa}}}{2} \! \left((\bm{y}^\mathrm{T}\Phi\bm{w})^2 - \|\Phi\bm{w}\|_2^2\right)
       + O \! \left(\lambda^{-\frac{3}{\kappa}}\right).
\end{align*}
The series expansion, combined with the condition $E_{p(\bm{w}|\lambda)}[\bm{w}] = \bm{0}$, gives
\begin{align*}
  &\int_{\mathbb{R}^m} p_0(\bm{w}) \exp\left(\lambda^{-\frac{1}{\kappa}} \bm{y}^\mathrm{T}\Phi\bm{w}-\frac{1}{2}\lambda^{-\frac{2}{\kappa}}\|\Phi\bm{w}\|_2^2\right) d\bm{w}\\
  &= 1 
    + \frac{\lambda^{-\frac{2}{\kappa}}}{2} \! \int_{\mathbb{R}^m} p_0(\bm{w}) \! \left((\bm{y}^\mathrm{T}\Phi\bm{w})^2 - \|\Phi\bm{w}\|_2^2\right) d\bm{w} 
    + O \! \left(\lambda^{-\frac{3}{\kappa}}\right).
\end{align*}
Utilizing the trace property $\bm{x}^\tp\bm{y}=\tr(\bm{x}\bm{y}^\tp)$ and the condition for $E_{p(\bm{w}|\lambda)}[\bm{w}\bm{w}^\mathrm{T}]$, 
the integral in the second term is manipulated as follows:
\begin{align*}
  &\int_{\mathbb{R}^m} p_0(\bm{w}) \left((\bm{y}^\mathrm{T}\Phi\bm{w})^2 - \|\Phi\bm{w}\|_2^2\right) d\bm{w}\\
  &= \int_{\mathbb{R}^m} p_0(\bm{w}) \tr\Bigl(\left((\bm{y}^\mathrm{T}\Phi)^\mathrm{T}(\bm{y}^\mathrm{T}\Phi) - \Phi^\mathrm{T}\Phi\right)\bm{w}\bm{w}^\mathrm{T}\Bigr) d\bm{w} \\
  &= \tr\left(\left((\bm{y}^\mathrm{T}\Phi)^\mathrm{T}(\bm{y}^\mathrm{T}\Phi) - \Phi^\mathrm{T}\Phi\right)\int_{\mathbb{R}^m} p_0(\bm{w}) \bm{w}\bm{w}^\mathrm{T}d\bm{w} \right)  \\
  &= \sigma_{\bm{w}}^2 \left(\|\Phi^\mathrm{T}\bm{y}\|_2^2 - \|\Phi\|_F^2\right).
\end{align*}
This completes the proof of Prop. \ref{prop: Z asymptotic expansion}.


\subsection{Proof of Lemma \ref{lemma: p(w) moment}} \label{sec: Proof of Lemma p(w) moment}
We first show the equality of $E_{p(\bm{w}|\lambda)}[\bm{w}]$ in Eq. \eqref{eq: condition of asymptotic expansion E[w]=0, E[ww]=sigma^2I}.
Under absolute homogeneity of $h$, $h(-\bm{x}) = |-1|^\kappa h(\bm{x}) = h(\bm{x})$. 
Applying the integral transformation $\tilde{\bm{w}} = -\bm{w}$ to $E_{p(\bm{w}|\lambda)}[\bm{w}]$ gives
\begin{align}
  E_{p(\bm{w}|\lambda)}[\bm{w}] 
  &= \int_{\setR^m} -\bm{w} \cdot 
     \frac{e^{-\lambda h(-\bm{w})}}{C}
     d\bm{w} \nonumber\\
  &= \int_{\setR^m} -\bm{w} \cdot 
     \frac{e^{-\lambda h(\bm{w})}}{C}
     d\bm{w} \nonumber\\
  &= - E_{p(\bm{w}|\lambda)}[\bm{w}] \label{eq-inproof: E_prior[w]},
\end{align}
where $C := \int_{\setR^m} e^{-\lambda h(\bm{w})}d\bm{w}$.
After transposing and rearranging the equation, we obtain $2E_{p(\bm{w}|\lambda)}[\bm{w}]= \bm{0}$, that is, $E_{p(\bm{w}|\lambda)}[\bm{w}]= \bm{0}$.

We next confirm the equality of $E_{p(\bm{w}|\lambda)}[\bm{w}\bm{w}^\mathrm{T}]$ in Eq. \eqref{eq: condition of asymptotic expansion E[w]=0, E[ww]=sigma^2I}.
Any matrix $A \in \mathcal{S}\cup\mathcal{P}$ satisfies $A^{-1} = A^T$, which means it is an orthogonal matrix.
Given the condition specified in Lemma 2, where $h(A\bm{x}) = h(\bm{x})$, a variable transformation $\tilde{\bm{w}} = A\bm{w}$
leads to the following equation, which is derived using a procedure analogous to that presented in Eq. \eqref{eq-inproof: E_prior[w]}:
\begin{align}
  E_{p(\bm{w}|\lambda)}[\bm{w}\bm{w}^\tp] 
  &= A^\tp E_{p(\bm{w}|\lambda)}[\bm{w}\bm{w}^\tp] A \label{eq-inproof:Eww^T = A^T・Eww^T・A}.
\end{align}
Hereafter, we divide $A$ into two patterns and evaluate the above equation.
\begin{enumerate}
  \item [(i)] 
  Consider the case where $A = S_{(i)} \in \mathcal{S}$, wherein $S_{(i)}$ is a diagonal matrix with the $i$-th diagonal element being $-1$ and all other diagonal elements being $+1$, and $i \in \{1,2,\ldots,m\}$.
  Note that $S_{(i)}^{-1} = S_{(i)}^\tp = S_{(i)}$.
  Plugging $S_{(i)}$ into Eq. \eqref{eq-inproof:Eww^T = A^T・Eww^T・A} yields
  \begin{align*}
    E_{p(\bm{w}|\lambda)}[\bm{w}\bm{w}^\tp] = S_{(i)} E_{p(\bm{w}|\lambda)}[\bm{w}\bm{w}^\tp] S_{(i)}
  \end{align*}
  Upon analyzing the rows and columns subjected to sign inversions by $S_{(i)}$, we derive
  \begin{align*}
    \left[E_{p(\bm{w}|\lambda)}[\bm{w}\bm{w}^\tp]\right]_{ij} 
    &= (-1)\left[E_{p(\bm{w}|\lambda)}[\bm{w}\bm{w}^\tp]\right]_{ij} \quad (j\neq i), 
  \end{align*}
  where $j \in \{1,2,\ldots,m\}$. 
  This results in
  \begin{align*}
    \left[E_{p(\bm{w}|\lambda)}[\bm{w}\bm{w}^\tp]\right]_{ij} 
    = 0 \quad (j\neq i).
  \end{align*}


  \item [(ii)]
  Consider the case where $A = P_{(i,j)} \in \mathcal{P}$, wherein $P_{(i,j)}$ is the matrix obtained by swapping the $i$-th and $j$-th columns of the identity matrix, and $i, j \in \{1,2,\ldots,m\}$.
  Note that $P_{(i,j)}^{-1} = P_{(i,j)}^\tp = P_{(i,j)}$.
  Inserting $P_{(i,j)}$ into Eq. \eqref{eq-inproof:Eww^T = A^T・Eww^T・A} gives
  \begin{align*}
    E_{p(\bm{w}|\lambda)}[\bm{w}\bm{w}^\tp] = P_{(i,j)} E_{p(\bm{w}|\lambda)}[\bm{w}\bm{w}^\tp] P_{(i,j)}.
  \end{align*}
  Inspecting the diagonal components replaced by $P_{(i,j)}$, we obtain
  \begin{align*}
    \left[E_{p(\bm{w}|\lambda)}[\bm{w}\bm{w}^\tp]\right]_{ii} 
    &= \left[E_{p(\bm{w}|\lambda)}[\bm{w}\bm{w}^\tp]\right]_{jj}.
  \end{align*}
\end{enumerate}
The arguments (i) and (ii) indicate that $E_{p(\bm{w}|\lambda)}[\bm{w}\bm{w}^\tp] = \sigma_{\bm{w}}^2 I$.


\end{document}